\documentclass[12pt]{amsart}
\usepackage{amsmath,amsthm,amssymb}
\usepackage{enumerate}
\usepackage[utf8]{inputenc}
\usepackage[english]{babel}
\usepackage{url}
\makeatletter
\@namedef{subjclassname@2010}{%
 \textup{2010} Mathematics Subject Classification}
\makeatother

\newtheorem{thm}{Theorem}[section]
\newtheorem{prop}[thm]{Proposition}
\newtheorem{cor}[thm]{Corollary}

\theoremstyle{definition}
\newtheorem{defin}[thm]{Definition}
\newtheorem{rem}[thm]{Remark}

\numberwithin{equation}{section}
\def\hh{{\mathcal H}}
\def\ll{{\mathcal L}}
\def\ss{{\mathcal S}}

\def\FF{\mathbb{F}}
\def\HH{\mathbb{H}}
\def\NN{\mathbb{N}}
\def\QQ{\mathbb{Q}}
\def\RR{\mathbb{R}}
\def\ZZ{\mathbb{Z}}
\def\be{{\mathbf e}}
\def\bu{{\mathbf u}}
\def\bv{{\mathbf v}}

\def\Norm{\mathop{\rm N}\nolimits}

\def\det{\mathop{\rm det}\nolimits}

\def\gcd{\mathop{\rm gcd}\nolimits}

\def\sgn{\mathop{\rm sgn}\nolimits}
\def\implies{{\>\Rightarrow\>}}

\begin{document}

%%%%% To ease editing, for IMPAN journals add:

% \baselineskip=17pt

\baselineskip=1.2\baselineskip

%% In the running head, replace first names by initials and give an
%% abbreviation of the title.

\title[Arithmetic and Geometry of Lipschitz Integers]{Some Connections
  Between The Arithmetic\\[3mm] and The Geometry of Lipschitz
  Integers}

\author[A. Machiavelo]{Ant\'onio Machiavelo} \address{Departamento de
  Matem\'atica\\ Faculdade de Ci\^encias da
  Universidade do Porto\\
  Rua do Campo Alegre, 4169-007 Porto \\ Portugal}
\email{ajmachia@fc.up.pt}

\author[L. Ro\c cadas]{Lu\'is Ro\c cadas}
\address{Centro de Matem\'atica, Universidade do Minho -- Polo CMAT-UTAD\\
  Departamento de Matem\'atica da Universidade de Tr\'as-os-Montes e
  Alto Douro\\ Quinta de Prados, 5001-801 Vila Real \\ Portugal}
\email{rocadas@utad.pt}

% \date{11/06/2011}
\date{\today}

%%%%%%%%%%%%%%%%%%%%%%%%%%%%%%%%%%%%%%%%%%%%%%%%%%%%%%%%%%%%%%%%%%%%%%%%

\begin{abstract}
  Some relationships between the arithmetic and the geometry of
  Lipschitz and Hurwitz integers are presented. In particular, it is
  shown that the (ternary) vector product of a Lipschitz integer
  $\alpha$ with two other Lipschitz integers, both orthogonal to
  $\alpha$, is a left and also a right multiple of $\alpha$, and that
  the vector product of two left multiples of $\alpha$ with any other
  Lipschitz integer is still a left multiple of $\alpha$. We also
  provide new arithmetical proofs for some old results of Gordon Pall,
  and raise a geometric problem on the location of some integral
  quaternions that is related to the factorization of some integers.
\end{abstract}

\subjclass[2010]{Primary 11R52; Secondary 11A51}

\keywords{Hurwitz integers, Lipschitz integers, ternary vector
  product, factorization}

\maketitle

%%%%%%%%%%%%%%%%%%%%%%%%%%%%%%%%%%%%%%%%%%%%%%%%%%%%%%%%%%%%%%%%%%%%%%%%
\section{Introduction}

The arithmetical properties of integral quaternions have been studied
since Lipschitz used them in 1886, in a paper on the real automorphs
of the form corresponding to the sum of three squares \cite{Lip}. Ten
years later, Hurwitz showed \cite{Hur96} that the integral quaternions
are contained in a slightly bigger ring that is arithmetically more
interesting, the ring of the now called Hurwitz integers, and expanded
their study in the 1919 monograph \cite{Hur19}.  In 1940, Gordon Pall
published an interesting paper on quaternion arithmetic \cite{Pal},
some results of which motivated the present work. The arithmetic of
quaternions continues to be investigated in more recent papers, namely
\cite{CoanPerng,abouzaid,CK15,forsyth}, and recently provided the
tools to solve an open conjecture on a refinement of Lagrange's four
squares theorem \cite{mac1,mac2}.

After describing, in the next section, the genesis behind the research
that led to the results contained in this paper, in section
\ref{sec:quat-lipsch-hurw} we summarize some of the basic facts on
quaternions, we recall the unique factorization theorem for Hurwitz
integers, and a result of Gordon Pall that offers some additional
information on factorizations, for which we provide a new, entirely
arithmetic proof. In section \ref{sec:orth-arithm-1} we collect some
results on relations between some arithmetical properties of integral
quaternions and orthogonality (in $\RR^4$). In particular, we prove a
simple result, Theorem \ref{thm:igama}, that yields some, albeit
remote and tantalizing, hope of an integer factorization method that
uses quaternions.  Finally, in section \ref{sec:vector-product-hh}, we
prove some divisibility results having to do with the (triple) vector
product. Namely, we show that the vector product of an integral
quaternion with two other that are orthogonal to it is both a left and
a right multiple of that quaternion, and that the vector product of
two left (resp.~right) multiples of an integral quaternion $\alpha$
with any other integral quaternion is also a left (right) multiple of
$\alpha$. We end the paper with a question naturally raised by this
last result.

%%%%%%%%%%%%%%%%%%%%%%%%%%%%%%%%%%%%%%%%%%%%%%%%%%%%%%%%%%%%%%%%%%%%%%%%
\section{Motivation}

Fr\'enicle de Bessy seems to have been the first to notice that one
can obtain a factorization of an integer $n$ from two different
decompositions of $n$ as a sum of two squares (\cite{Dic}, vol.~I,
cap.~XIV, p.~360). This amounts to the fact that a decomposition
$n=a^2+b^2$ gives a factorization $n=(a+bi)(a-bi)$ in $\ZZ[i]$, and if
one has another decomposition $n=c^2+d^2$, then, using the Euclidean
algorithm in $\ZZ[i]$, one can compute the greatest common divisor of
$a+bi$ and $c+di$, whose norm yields a non-trivial factor of $n$.

As Bachet de M\'eziriac conjectured and Lagrange proved, every number
is a sum of four squares (\cite{Dic}, vol.~II, cap.~VIII,
p.~275). Now, while there is no known fast algorithm to decompose a
number as a sum of two squares, there is a very efficient
probabilistic algorithm, due to Rabin and Shallit \cite{ReS}, to
express a number as a sum of four squares.  Such a decomposition of an
integer $n$ yields a factorization $n=\alpha\bar{\alpha}$ in the ring
of Hurwitz integers. Since this ring is both a left and a right
Euclidean domain, it is natural to wonder if two distinct
decompositions of a number as a sum of four squares could yield a
factorization of that number in a manner analogous to what happens in
$\ZZ[i]$.

However, if one has two essentially distinct factorizations of $n$,
$n=\alpha\bar{\alpha}=\beta\bar{\beta}$, it is not always the case
that $\alpha$ and $\beta$ will have a non-trivial left or right
greatest common divisor. In fact, for a number that is a product of
two odd primes, $n=pq$, only a small (for $p$ and $q$ big) fraction,
precisely $\frac{p+q+2}{(p+1)(q+1)}$, of all possible pairs
$\alpha,\beta$ will have a (left or right) greatest common divisor
whose norm is neither $1$ nor $n$. But, in \cite{Pal}, Gordon Pall
proves a series of interesting results (namely, Theorems 6 and 7),
which imply, in particular, that given two quaternions $\alpha$ and
$\beta$ with integral coprime coordinates, if they are orthogonal and
have the same norm, then they either have the same right divisors, or
the same left divisors, or both. This suggests looking for orthogonal
decompositions of an integer as a sum of four squares, i.e.~orthogonal
integral quaternions whose norm is that integer. If one could find
some efficient way to find such decompositions, one would hope to get
an interesting factorization algorithm.

It was this line of thought that made us study ways of constructing
quaternions that are orthogonal to a given quaternion, namely using
the ternary vector product, and that led to the discovery of the main
results here presented, namely Theorems \ref{thm:vp_2ort} and
\ref{thm:vp_2mult}. Although these results are negative for the
purposes mentioned above, we believe that they are interesting in
their own way, showing some intimate connections between geometry and
arithmetic in the realm of quaternions.

%%%%%%%%%%%%%%%%%%%%%%%%%%%%%%%%%%%%%%%%%%%%%%%%%%%%%%%%%%%%%%%%%%%%%%%%
\section{Quaternions, Lipschitz and Hurwitz integers}
\label{sec:quat-lipsch-hurw}

We start by recalling that the quaternion ring $\HH$ is the division
ring consisting of the additive group $\RR^4$ endowed with the only
multiplication (so one gets a ring structure) determined by choosing
$\be_1 =1$, the multiplicative unit, and by the relations:
$$
\be_2^2=\be_3^2=\be_4^2= \be_2\, \be_3\, \be_4=-\be_1 = -1,
$$
where $\{\be_1, \be_2, \be_3, \be_4\}$ is the canonical basis of
$\RR^4$. Usually, in this context, one denotes the elements of this
basis by $1,i,j,k$, respectively. One can easily check that, then:
\begin{multline}\label{produto}
    (u_0 + u_1 \,i + u_2 \,j + u_3 \,k)(v_0 + v_1 \,i + v_2 \,j + v_3
    \,k)=\hphantom{mmmmm}\\
    \hphantom{mmmmmmm} = (u_0 v_0-u_1 v_1-u_2 v_2-u_3 v_3)+
    (u_0 v_1+u_1 v_0+u_2 v_3-u_3 v_2) \,i \\
    + (u_0 v_2-u_1 v_3+u_2 v_0+u_3 v_1) \,j+ (u_0 v_3+u_1 v_2-u_2
    v_1+u_3 v_0) \,k.
  \end{multline}

Given a quaternion
$u=a+b i+c j+d k$, its {\em conjugate} is defined by
$\bar{u}=a-b i-c j-d k$, and its {\em norm} is $\Norm(u)=u\bar{u}$. We
set $\Re(u)=a$, the \emph{real} part of $u$, and $\Im(u)=b i+c j+d k$,
the \emph{imaginary} or \emph{vector} part of $u$.

The quaternions with integral coordinates are called {\em Lipschitz
  integers}, and they form a subring of $\HH$ that we will denote by
$\ll$. This is almost a left (and right) Euclidean ring for the norm,
in the sense that for any $\alpha,\beta\in\ll$ one can find
$q,r\in\ll$ such that $\alpha = \beta q + r$ and
$\Norm(r)\leq\Norm(\beta)$, but a strict inequality cannot always be
guaranteed (and the same for right division). However, one needs only
to slightly enlarge $\ll$ by adding the quaternions whose coordinates
are all halves of odd numbers to obtain a (left and right) Euclidean
ring. This yields the set $\hh=\ll\cup\left( \omega+\ll\right)$, with
$\omega=\frac12(1+i+j+k)$, whose elements are called {\em Hurwitz
  integers}. One can easily show that any Hurwitz integer has both a
left and a right associate which is a Lipschitz integer (see
\cite[Lemma 11.2.9]{Voight}).

The euclidianity of $\hh$ implies that every left and every right
ideal of $\hh$ is principal, and from this a sort of unique
factorization into \emph{primes}, Hurwitz integers whose norm is a
rational prime, can be deduced for {\em primitive} Hurwitz integers,
i.e.~those not divisible by a rational prime.

\begin{thm}[Unique Factorization Theorem]\label{uniqueF}
  Corresponding to each factorization of the norm $n$ of a primitive
  Hurwitzian integer $\alpha$ into a product $p_1 p_2 \cdots p_{k-1}
  p_k$ of rational primes, there is a factorization $$\alpha =\pi_1
  \pi_2\cdots \pi_{k-1} \pi_k$$ of $\alpha$ into a product of
  Hurwitzian primes that is said to be {\em modelled on} that
  factorization of $n$, that is, with $\Norm(\pi_i)=p_i$.

  Moreover, if $\alpha =\pi_1 \pi_2\cdots \pi_{k-1} \pi_k$ is any one
  factorization modelled on $p_1 p_2 \cdots p_{k-1} p_k$, then all the
  others have the form $$\alpha =\pi_1 \varepsilon_1\cdot
  \varepsilon_1^{-1} \pi_2 \varepsilon_2\cdot \varepsilon_2^{-1} \pi_2
  \varepsilon_3\cdot\> \cdots\> \cdot \varepsilon_{k-2}^{-1} \pi_{k-1}
  \varepsilon_{k-1}\cdot \varepsilon^{-1}_{k-1} \pi_k,$$ where
  $\varepsilon_1, \ldots, \varepsilon_{k-1}\in\hh^*$, i.e.~the
  factorization on a given model is unique up to unit-migration.
\end{thm}

This result is essentially contained in \cite{Lip} (p.~434), where
Lipschitz proves that integral quaternions have that same sort of
unique factorization up to factors of norm $2$. For a modern proof see
Theorem 2, p.~57 in \cite{CeS}.

Given $m\in\NN$, a quaternion $\alpha=a+b i+c j+d k\in\ll$ is said to
be \emph{primitive modulo} $m$ if $\gcd(a,b,c,d,m)=1$.  In
\cite[Theorem 1]{Pal}, Pall proves the following result, using some
classical results about quadratic forms, and of which we give here a
completely arithmetical proof.

\begin{thm}\label{PallTh1}
  If $\alpha\in\ll$ is primitive modulo $m$, where $m$ is odd and
  positive with $m \mid \Norm(\alpha)$, then $\alpha$ has a unique, up
  to left associates, right divisor of norm m, in $\ll$. One has an
  analogous result for left divisors.
\end{thm}

\begin{proof}
  Since every left ideal of $\hh$ is principal, and every Hurwitz
  integer has a Lipschitz associate, there exists $\delta\in\ll$ such
  that $\hh \alpha + \hh m = \hh\delta$. In particular,
  $\delta = \beta\alpha+\gamma m$, for some $\beta,\gamma\in\hh$. But
  then
  $\Norm(\delta) = \Norm(\alpha)\Norm(\beta) +
  2\Re(\beta\alpha\bar{\gamma})m+\Norm(\gamma) m^2$, and thus
  $m\mid \Norm(\delta)$. Let $t\in\NN$ be the respective quotient, so
  that $\Norm(\delta) = mt$, and let $\sigma,\tau\in\hh$ be such that
  $\alpha = \sigma\delta$ and $m=\tau \delta$. Then
  $m^2=\Norm(\tau) \Norm(\delta) = \Norm(\tau)mt$ shows that
  $\tau\delta = m = \tau\bar{\tau} t$, and thus
  $\delta = \bar{\tau} t$. The fact that $\alpha$ is primitive modulo
  $m$ now entails $t=1$, showing the existence of a right divisor of
  $\alpha$ with norm $m$ in $\hh$. It remains to show that
  $\sigma\in\ll$. Since $m$ is odd, there are $x,y\in\ZZ$ such that
  $2x+my=1$. But then
  $\sigma = 2\sigma x + \alpha \bar{\delta} y\in\ll$.
  
  To prove uniqueness, up to left associates, assume that
  $\alpha = \xi \mu$ for some $\xi,\mu\in\ll$, with
  $\Norm(\mu)=m$. Then
  $\delta = \beta\alpha + \gamma m = (\beta\xi + \gamma
  \bar{\mu})\mu$, showing that $\mu$ is a right divisor of
  $\delta$. Since they have the same norm, one has
  $\epsilon := \beta\xi + \gamma \bar{\mu}\in \hh^*$. Using again the
  fact that $m$ is odd, one shows that $\epsilon\in\ll^*$ by noticing
  that
  $\epsilon = 2\epsilon x + \epsilon my = (2\epsilon) x + (\epsilon
  \mu) \bar{\mu}y = (2\epsilon) x + \delta \bar{\mu}y $.
\end{proof}

\noindent\textbf{Remarks:}
\begin{itemize}
\item This last result does not hold for $m$ even, as the following
  example shows: $1+i+j+k = (1+i)(1+j) = (1+k)(1+i)$, while $1+j$ and
  $1+i$ are not left associates in $\ll$. But it is very easy to see
  that the result does inconditionally hold in $\hh$.
\item The map $\hh\to\hh$ given by $\alpha\mapsto\bar{\alpha}$ is an
  anti-automorphism, and so any divisibility result on the left also
  holds on the right.
\end{itemize}

\medskip

Notice that while Theorem~\ref{uniqueF} relates factorizations
modelled on the same prime decomposition of the norm,
Theorem~\ref{PallTh1} gives information about factorizations of a
primitive quaternion modelled on different prime decomposition of its
norm. For example, if $\alpha =\pi_1 \pi_2 \pi_3$ is a factorization
of a primitive quaternion $\alpha$ corresponding to
$\Norm(\alpha)=p_1 p_2 p_3$, and $\alpha =\pi'_2 \pi'_1 \pi'_3$ is a
factorization corresponding to $\Norm(\alpha)=p_2 p_1 p_3$, then it
follows from the last theorem that $\pi_1 \pi_2$ and $\pi'_2 \pi'_1$
are right associates, and therefore $\pi_3$ and $\pi'_3$ are left
associates. It turns out that the version of Theorem~\ref{PallTh1} for
Hurwitz integers implies Theorem~\ref{uniqueF}, as it is fairly easy
to see.
%%%%%%%%%%%%%%%%%%%%%%%%%%%%%%%%%%%%%%%%%%%%%%%%%%%%%%%%%%%%%%%%%%%%%%%%
\section{Orthogonality and Arithmetic}
\label{sec:orth-arithm-1}

From the expression \eqref{produto} above, that gives the product of
two generic quaternions, $u$ and $v$, it immediately follows that, for
the inner product $u\cdot v$ (as vectors of $\RR^4$), one has:
\begin{equation}
  \label{innerp}
  u\cdot v = \Re(u\bar{v}) =  \frac{1}{2}
  \left( u\bar{v} + v\bar{u}\right).
\end{equation}
Then, for all $u,v,\alpha\in\HH$,
\begin{equation}
  \label{innerpl}
  u\alpha\cdot v = \Re((u\alpha)\bar{v})= \Re(u(\alpha\bar{v}))= u\cdot
  \overline{\alpha\bar{v}} = u\cdot v\bar{\alpha},
\end{equation}
and using the obvious fact that $u\cdot v=\bar{u}\cdot \bar{v}$, one
also has:
\begin{equation}
  \label{innerpr}
  \alpha u\cdot v = \bar{u}\bar{\alpha}\cdot \bar{v}= \bar{u}\cdot
  \bar{v} \alpha = u\cdot \bar{\alpha} v.
\end{equation}
Let us also point out that the expression \eqref{produto} also yields:
\begin{equation}
  \label{uv}
  uv= \Re(v) u + \Re(u) v - u\cdot v + \Im(u)\times \Im(v),
\end{equation}
where $\times$ denotes the usual vector product in $\RR^3$, after
using the natural identification of pure quaternions with
$3$-dimensional vectors, or equivalently, by setting
$$
(u_1 i+u_2 j+u_3 k)\times (v_1 i+v_2 j+v_3 k)= \left|
  \begin{matrix}
    u_1 &  u_2 &  u_3  \\
    v_1 & v_2 & v_3 \\
    i & j & k
  \end{matrix}
\right| $$
with the obvious meaning.

In what follows, we will use the notation $u\perp v$ to mean that the
quaternions $u$ and $v$ are orthogonal, i.e.~$u\cdot v=0$.  In
\cite{Pal}, Pall shows that there are interesting connections between
arithmetic properties of Lipschitz integers and orthogonality. We here
exhibit some others, and provide a simpler arithmetical proof for a
particular case of a result of Pall.

\begin{prop}\label{sameRDinnerp}
  For any $u, v, w\in\HH$, one has 
  $$(uv)\cdot(uw) = \Norm(u)\> (v\cdot w).$$
  In particular, if $\alpha, \beta\in\ll$ have a common left divisor
  $\tau$, then $\Norm(\tau) \mid \alpha\cdot\beta$. One has analogous
  results for right common divisors.
\end{prop}

\begin{proof} 
  This is an immediate consequence of \eqref{innerpr}.
\end{proof}

\begin{cor}\label{ijkOrtho}
  Let $\epsilon,\delta\in\{1,i,j,k\}$ with $\epsilon\not=\delta$.
  Then, for any $\alpha\in\HH$, $\alpha\epsilon\perp \alpha\delta$ and
  $\epsilon\alpha\perp \delta\alpha$.
\end{cor}

\begin{proof} This is an immediate consequence of the previous
  proposition, and the fact that $\epsilon\perp \delta$. \end{proof}

It follows from Theorem 6 in \cite{Pal} that two non-associate
Hurwitzian primes that have the same norm cannot be orthogonal. We
show here that this can be directly deduced from the unique
factorization theorem.

\begin{thm}
  If $\alpha, \beta\in\hh$ are primes with the same norm, and
  $\alpha\perp \beta$, then each one is a left, as well as a right
  associate of the other.
\end{thm}

\begin{proof} Let $p=\Norm(\alpha)=\Norm(\beta)$. From
  $\alpha\perp\beta$ one gets that
  $\alpha\bar{\beta}=-\beta\bar{\alpha}$. Now, if the quaternion
  $\gamma=\alpha\bar{\beta}$ is not primitive, then $m\mid\gamma$ for
  some $m\in\NN$ with $m>1$. But then, from
  $m^2\mid\Norm(\gamma)=p^2$, it follows that $m=p$. But then
  $\alpha\bar{\beta} = p \varepsilon=\varepsilon p$, for some unit
  $\varepsilon$.  Since $p=\beta\bar{\beta}$, one gets
  $\alpha=\varepsilon\beta$. From
  $\beta\bar{\alpha}=-\alpha\bar{\beta}= -p \varepsilon$, one gets
  $\beta=-\varepsilon\alpha$ (and in this case one sees that
  $\varepsilon^2=-1$, and therefore $\varepsilon=\pm i, \pm j,\pm k$).

  If $\gamma$ is primitive, then $\alpha\bar{\beta}$ and
  $-\beta\bar{\alpha}$ are two factorizations of $\gamma$ modelled on
  $\Norm(\gamma)= p p$, and the unique factorization theorem implies
  that $\alpha$ and $\beta$ are right associates.

  Finally note that $\alpha\perp \beta\implies \bar{\alpha}\perp
  \bar{\beta}$, which allows to deduce the left version of the result
  from its right version, and vice-versa.\end{proof}

With non-primes one can obtain examples that are a little more
interesting.  For instance, from the previous corollary, it follows
that if $\pi$ and $\rho$ are any two quaternions, then $\pi i\rho$ and
$\pi\rho$ have the same norm and are orthogonal. The question of what
exactly is the left greatest common divisor of these two quaternions,
leads to:

\begin{prop}
  \label{thm:igama}
  Let $\gamma=z+w j\in\ll$, with $z, w\in\ZZ[i]$, be an odd quaternion
  {\rm(}i.e.~$\gamma$ has an odd norm{\rm)}. Then:
  $$i\gamma\hh+\gamma\hh =1\quad \iff\quad (z,w)=1 \quad 
  (\text{in } \ZZ[i]).$$
\end{prop}

\begin{proof} Set $I= i\gamma\hh+\gamma\hh$. It is clear that if
  $\delta\mid z$ and $\delta\mid w $, with $\delta\in\ZZ[i]$, then
  $\delta$ is a left divisor of $i\gamma$, since of course
  $\delta\mid\gamma$ and it commutes with $i$. Therefore,
  $(z,w)=(\delta)$ implies that $ I\subseteq \delta\hh$.

  On the other hand, since $i\gamma=z i+w k$, $\gamma i=z i-w k$, one
  has:
  $$2z i = i\gamma+\gamma i\in I$$
  and
  $$2w k = i\gamma-\gamma i\in I.$$
  Hence: $$2z, 2w\in I.$$ Now, $(2,\Norm(\gamma))=1$ implies
  that there are $x,y\in\ZZ$ such that $2x+\gamma\bar{\gamma}y=1$. In
  particular, there is $x\in\ZZ$ with $2 x\equiv 1\pmod I$. From this
  one concludes that $z, w\in I$, and so, if these are
  coprime, it follows that $I=1$.
\end{proof}

Note that from an algorithm to compute $\pi i\rho$ from the quaternion
$\pi\rho$ one would get a factorization algorithm for some integers,
namely semi-primes. In fact, suppose we have a semi-prime number
$n=pq$ with $p$ and $q$ to be determined. Using an algorithm like the
one in \cite{ReS}, one can find $\alpha\in\ll$ such that
$\Norm(\alpha)=n$, and one has $\alpha=\pi\rho$, for some primes
$\pi,\rho\in\hh$. If one could determine $\pi i\rho$ from $\alpha$,
then using the Euclidean algorithm, one would get $\pi$ and $\rho$,
since the previous result easily implies that
$\pi i\rho\hh+\pi\rho\hh = \pi\hh$ (and, in an entirely analogous way,
$\hh\pi i\rho+\hh\pi\rho = \hh\rho$). This would yield $p$ and $q$. In
order to get an interesting factoring algorithm along these lines it
would, of course, be enough to find a method of determininig a
reasonable sized neighborhood in the orthogonal space to $\alpha$
where $\pi i\rho$ would be located. Given that integer factorization
seems to be a very hard problem, it is to be expected that the
relation between the coordinates of $\pi\rho$ and the ones of
$\pi i\rho$ will be rather subtle.

We leave here just one example that illustrates the seemling lack of
relation between the coordinates of $\pi\rho$ and $\pi i \rho$, as
well as the natural variations:
{\small \baselineskip=1.2\baselineskip
$$
\begin{array}{rcl|rcl}
  \pi &=& 1+i+3j+6k \quad{\scriptstyle (N(\pi) = 47)}
  & \rho &=& 1+2i+5j+7k \quad{\scriptstyle (N(\rho) = 79)}\\[2mm]
  \pi\rho &=& -58-6i+13j+12k & \rho\pi &=& -77+4i+10j+14k \\[2mm]
  \pi i\rho &=& -12+56-12j-17k & \rho i\pi &=& 6+56i-10j-21k \\[2mm]
  \pi j\rho &=& -3-10i+30j-52k & \rho j\pi &=& -13-12i+30j-50k \\[2mm]
  \pi k\rho &=& -14-21i-50j-24k & \rho k\pi &=& -12-17i-52j-54k
\end{array}
$$
}

We end this section by showing that one can easily get a $\ZZ$-basis
for the $\ZZ$-module of the integral quaternions that are orthogonal
to a given primitive integral quaternion.

\begin{prop}\label{orto}
  Let $\alpha=a+b i+c j+d k\in\ll$ be a primitive quaternion. Let
  $g_1,g_2\in\ZZ$ be such that $g_1\ZZ=a\ZZ+b\ZZ$, $g_2\ZZ=c\ZZ+d\ZZ$,
  and $x_0,y_0,z_0,t_0\in\ZZ$ be such that: $a x_0+b y_0 = g_1$,
  $c z_0+d t_0 = g_2$. In case $g_1=0$, we choose $x_0=0$ and $y_0=1$,
  and similarly if $g_2=0$. We make the convention that
  $\frac{0}{0}+\frac{0}{0}i= 1$.  Then the $\ZZ$-module
  $\alpha^\perp\cap\ll$ is generated by the quaternions:
  $$g_2 (x_0+ y_0 i)-g_1(z_0 + t_0 i)j, 
  \quad \frac{b}{g_1} - \frac{a}{g_1} i,\quad \left(\frac{d}{g_2} -
    \frac{c}{g_2} i\right)j$$
\end{prop}

\begin{proof}
  Suppose we are given $\alpha=a+b i+c j+d k\in\ll$, primitive.  We
  want to find all vectors in $\ll\cap\alpha^\perp$. Let $g_1, g_2$ be
  as in the statement above, and assume first that $g_1
  g_2\not=0$. Let $x_0,y_0,z_0,t_0\in\ZZ$ be such that:
  \begin{eqnarray}
    a x_0+b y_0 &=& g_1\label{eql1}\\
    c z_0+d t_0 &=& g_2\label{eql2}
  \end{eqnarray}
  Now, any quaternion $\gamma = x+y i+z j+t k\in\ll$ such that
  \begin{equation*}
    a x+b y+c z+d t=0. 
  \end{equation*}
  must satisfy
  \begin{eqnarray*}
    a x+b y&=& r_1 g_1\\
    c z+d t&=& r_2 g_2,
  \end{eqnarray*}
  for some $r_1,r_2\in\ZZ$ with $r_1 g_1+r_2 g_2=0$. Because $\alpha$
  is primitive, $g_1\ZZ+g_2\ZZ=1$, and therefore there exists
  $r\in\ZZ$ such that $r_1=r g_2$ and $r_2=-r g_1$. It then follows
  from the well known caracterization of the solutions of linear
  Diophantine equations that:
  \begin{equation}\label{xyzt}
    \begin{split}
      x &= r_1 x_0 + \frac{b}{g_1} s = r g_2 x_0 + \frac{b}{g_1} s\\
      y &= r_1 y_0 - \frac{a}{g_1} s = r g_2 y_0 - \frac{a}{g_1} s\\
      z &= r_2 z_0 + \frac{d}{g_2} u = -r g_1 z_0 + \frac{d}{g_2} u\\
      t &= r_2 t_0 - \frac{c}{g_2} u = -r g_1 t_0-\frac{c}{g_2} u,
    \end{split}
  \end{equation}
  for some $s,u\in\ZZ$.

  It is easy to check that the result holds in the two cases in which
  $g_1 g_2=0$, if one uses the conventions formulated in the statement
  of the lemma.
\end{proof}

%%%%%%%%%%%%%%%%%%%%%%%%%%%%%%%%%%%%%%%%%%%%%%%%%%%%%%%%%%%%%%%%%%%%%%%%
\section{The vector product in $\HH$ and the arithmetic of $\ll$}
\label{sec:vector-product-hh}

We show in this section that some triple vector products of some
quaternions involving a given quaternion $\alpha$, being orthogonal to
$\alpha$, are nevertheless multiples of $\alpha$. The results are
really of an algebraic nature, in the sense that they follow from some
polynomial identities, which one can (implicitly) verify using, for
instance, SageMath \cite{sage}. Albeit the Sage verification being a
proof of the following two theorems, it not a very enlightning one. We
provide proofs that we believe to be interesting in their own right.

We start by recalling the notion of vector product in $\RR^n$.

\begin{defin}\label{defgenvp}
  For $\bu_1, \bu_2, \ldots, \bu_{n-1}\in\RR^n$, define their vector
  product by
  $$\times(\bu_1, \bu_2, \ldots, \bu_{n-1}) = \bu_1\times
  \bu_2\times\cdots\times \bu_{n-1} := \left|
    \begin{matrix}
      u_{1,1} &  u_{1,2}  & \cdots &  u_{1,n}  \\
      \vdots  &  \vdots   & \cdots &  \vdots   \\
      u_{n-1,1} & u_{n-1,2} & \cdots &  u_{n-1,n}\\
      \be_1 & \be_2 & \cdots & \be_n
    \end{matrix}
  \right|
  $$
  (with the obvious meaning, using Laplace expansion on the last row),
  where ${\mathbf e_i}$ is the $i$-th vector of the canonical basis of
  $\RR^n$, and $u_{i,j}$ is the $j$-th coordinate of the vector
  $\bu_i$, on that same basis.
\end{defin}

It immediately follows from this definition that, for any vectors
$\bu_i, \bv\in\RR^n$, where $i=1,\ldots,n-1$, one has:
\begin{equation}
  \label{eq:vecinner}
  (\bu_1\times \bu_2\times \cdots\times \bu_{n-1})\cdot
  \bv = \left|
    \begin{matrix}
      u_{1,1} &  u_{1,2}  & \cdots &  u_{1,n}  \\
      \vdots  &  \vdots   & \cdots &  \vdots   \\
      u_{n-1,1} & u_{n-1,2} & \cdots &  u_{n-1,n}\\
      v_1   &     v_2   & \cdots &   v_n     \\
    \end{matrix}
  \right| 
\end{equation}
\begin{equation}
  \label{eq:vecperp}
  (\bu_1\times \bu_2\times \cdots\times \bu_{n-1})\perp \bu_i, \text{
    for all }  i=1,\ldots, n-1.
\end{equation}
Also, for $\alpha,\beta,\gamma,\delta\in\HH$, it is clear that:
\begin{align}
  \overline{\alpha\times\beta\times\gamma}
  &= -\,\bar{\alpha}\times\bar{\beta}\times \bar{\gamma}\, ,
    \label{eq:vpconj}\\ 
  (\alpha\times\beta\times\gamma)\cdot\delta
  &= - (\alpha\times\beta\times\delta)\cdot\gamma\, , \\
  \alpha\times\beta\times 1
  & = \Im(\alpha) \times
    \Im(\beta) = \frac12 (\alpha\beta - \beta\alpha) \text{ is a pure
    quaternion}, \label{eq:vpx1}
\end{align}
the last equality following from~\eqref{uv}.

\begin{prop}
\label{innervector}
  For any vectors $\bu_i, \bv_j\in\RR^n$, with
  $i,j=1,\ldots,n-1$, one has:
  $$  (\bu_1\times \bu_2\times \cdots\times \bu_{n-1})\cdot (\bv_1\times
    \bv_2\times \cdots\times \bv_{n-1}) = \det(\bu_i\cdot
    \bv_j).
  $$
\end{prop}

\begin{proof}
  Since the two maps from $\left(\RR^n\right)^{n-1}$ to $\RR$ given by
  $$(\bu_1, \bu_2, \ldots, \bu_{n-1})\to (\bu_1\times \bu_2\times
  \cdots\times \bu_{n-1}) \cdot (\bv_1\times \bv_2\times \cdots\times
  \bv_{n-1})$$ and by
  $$(\bu_1, \bu_2,\ldots, \bu_{n-1})\to \det(\bu_i\cdot \bv_j)$$
  are both multilinear, it is enough to check the validity of the
  claimed result for $\bu_i, \bv_j\in\{\be_1, \be_2, \ldots,
  \be_n\}$. And because the mentioned maps are also alternate, it is
  enough to check that:
  $$
  (\be_{\sigma(1)}\times \be_{\sigma(2)}\times \cdots\times
  \be_{\sigma(n-1)}) \cdot (\be_{\tau(1)}\times \be_{\tau(2)}\times
  \cdots\times \be_{\tau(n-1)}) = \det(\be_{\sigma(i)}\cdot
  \be_{\tau(j)}),
  $$
  for all $\sigma, \tau\in \ss_n$, where $\ss_n$ denotes the symmetric
  group on $\{1,\ldots,n\}$.

  Now,
  $$\be_{\sigma(1)}\times \cdots\times \be_{\sigma(n-1)} =
  \sum\limits_{\tau\in\ss_n} \sgn(\tau)\, e_{\sigma(1)\tau(1)} \cdots
  e_{\sigma(n-1)\tau(n-1)}\, \be_{\tau(n)},$$ where, as in the notation
  used in Definition \ref{defgenvp}, $e_{ij}$ denotes the $j$-th
  coordenate of $\be_i$. The only non-zero term of this sum is the one
  where $\tau= \sigma$. It follows that:
  $$
  \be_{\sigma(1)}\times \be_{\sigma(2)}\times \cdots\times
  \be_{\sigma(n-1)} = \sgn(\sigma)\, \be_{\sigma(n)}.
  $$
  Therefore:
  $$
  (\be_{\sigma(1)}\times \be_{\sigma(2)}\times \cdots\times
  \be_{\sigma(n-1)}) \cdot (\be_{\tau(1)}\times \be_{\tau(2)}\times
  \cdots\times \be_{\tau(n-1)}) = \sgn(\sigma\tau)\,
  \delta_{\sigma(n),\tau(n)}.
  $$
  On the other hand,
  $$\det(\be_{\sigma(i)}\cdot \be_{\tau(j)})=
  \sum\limits_{\gamma\in\ss_{n-1}} \sgn(\gamma) (\be_{\sigma(1)}\cdot
  \be_{\tau(\gamma(1))}) \cdots (\be_{\sigma(n-1)}\cdot
  \be_{\tau(\gamma(n-1))}),$$
  which is non-zero only for
  $\gamma= \tau^{-1}\sigma_{|\{1,\ldots,n-1\}}$, and this happens only
  when $\sigma(n)=\tau(n)$. It easily follows that:
  $$
  \det(\be_{\sigma(i)}\cdot \be_{\tau(j)})= \sgn(\tau^{-1}\sigma)\,
  e_{\sigma(n)\tau(n)},
  $$
  proving what we wanted to show.
\end{proof}

From this proposition, one sees that, for any
$\alpha, \beta, \gamma\in \HH$,
$$\Norm(\alpha\times \beta\times \gamma) = (\alpha\times \beta\times
\gamma)\cdot(\alpha\times \beta\times \gamma)= \left|
\begin{matrix} 
  \Norm(\alpha)   & \alpha\cdot \beta & \alpha\cdot \gamma  \\
  \alpha\cdot \beta &    \Norm(\beta)  & \beta\cdot \gamma  \\
  \alpha\cdot \gamma & \beta\cdot \gamma &    \Norm(\gamma)   \\
\end{matrix}
\right|, $$ from which one easily gets
\begin{multline*}
  \Norm(\alpha\times \beta\times \gamma)=
  \Norm(\alpha\beta\gamma)-\Norm(\alpha)(\beta\cdot \gamma)^2 -
  \Norm(\beta)(\alpha\cdot \gamma)^2 -\\
  - \Norm(\gamma)(\alpha\cdot \beta)^2+2(\alpha\cdot \beta)
  (\alpha\cdot \gamma) (\beta\cdot \gamma).
\end{multline*}

In particular, if $\beta\perp \alpha$ and $\gamma\perp \alpha$, then
$\Norm(\alpha)\mid \Norm(\alpha\times \beta\times \gamma)$. It follows
from Theorem \ref{PallTh1} that, in this case,
$\alpha\times \beta\times \gamma$ has both a left and a right divisor
with the same norm as $\alpha$. We will show that, in both cases, it
turns out that $\alpha$ is that divisor.

\begin{thm}
  \label{thm:vp_2ort}
  Given $\alpha\in\ll$, and $\beta, \gamma\in\ll$ such that
  $\beta\perp\alpha$ and $\gamma\perp\alpha$, one has
  $$\alpha\times\beta\times\gamma\in\alpha\ll\cap\ll\alpha.$$
\end{thm}

\begin{proof} 
  Let $\alpha=a+b i+c j+d k\in\ll\setminus\{0\}$, and assume that
  $d\neq 0$ (if not, the following argument still works \emph{mutatis
    mutandis}). Then an $\QQ$-basis for $\alpha^\perp$ is given by
  $\beta_1= d-a k, \beta_2= d i-b k, \beta_3= d j-c k$. Now, simple
  computations yield:
  \begin{eqnarray*}
    \alpha\times\beta_1\times\beta_2
    &=& -\alpha\,\Im(j\alpha)\,d = -d\,\Im(\alpha j)\,\alpha,\\
    \alpha\times\beta_1\times\beta_3
    &=&  \alpha\,\Im(i \alpha) \, d = - d\, \Im(\alpha i)\,\alpha,\\
    \alpha\times\beta_2\times\beta_3
    &=& \alpha\,\Im(\alpha)\, d = d\, \Im(\alpha)\,\alpha,
  \end{eqnarray*}
  which, by the multilinearity of the vector product, proves the claim.
\end{proof}

Using corollary \ref{ijkOrtho}, one sees that, for example,
$\alpha\times \alpha i\times \alpha j\in\alpha\ll\cap\ll\alpha$. While
doing some computational experiments, we noticed that, for example,
$\alpha\times \alpha i\times \beta \in\alpha\ll$, for all
$\beta\in\ll$. This eventually led to the discovery of the next
results that connect the vector product with the multiplication of
quaternions.

\begin{thm}
  \label{vectprodN}
  For any $\alpha,\beta,\gamma,\delta\in\HH$, one has:
  \begin{align*}
    \alpha\beta\times \alpha\gamma\times \alpha\delta
    &=
      \Norm(\alpha)\,\alpha\,(\beta\times\gamma\times\delta),\\
    \beta\alpha\times \gamma\alpha \times \delta\alpha
    &= \Norm(\alpha)\,
      (\beta\times\gamma\times\delta)\,\alpha.
  \end{align*}
\end{thm}

\begin{proof}
  We will show the first equality, the proof of the second being
  entirely analogous.

  Recall (see \cite[Remark 3.3.8]{Voight}) that the determinant of the
  left regular representation $\HH\to\HH$ given by $x\mapsto \alpha x$
  is equal to $\Norm(\alpha)^2$, and note that by multilinearity and
  alternatingness it is enough to show the claimed equality for
  $\beta, \gamma, \delta\in\{1,i,j,k\}$.  In order to do that, let
  $\varepsilon_i$, $i=1,2,3,4$, be distinct units such that
  $\varepsilon_1,\varepsilon_2,\varepsilon_3 \in \{1,i,j,k\}$, and set
  $\varepsilon_4 =\varepsilon_1\times
  \varepsilon_2\times\varepsilon_3$ Then
  $\{\varepsilon_1,\varepsilon_2,\varepsilon_3 ,\varepsilon_4\}$ is an
  $\RR$ basis for $\HH$, and $\Norm(\alpha)^2$ is the value of the
  determinant whose rows are the coordinates of $\alpha\varepsilon_i$,
  $i=1,2,3,4$, which is equal to
  $(\alpha\,\varepsilon_1\times \alpha\,\varepsilon_2\times
  \alpha\,\varepsilon_3) \cdot \alpha\,\varepsilon_4$. Using
  \eqref{innerpl}, one then has
  $\Norm(\alpha)^2 = \bar{\alpha}(\alpha\,\varepsilon_1\times
  \alpha\,\varepsilon_2\times \alpha\,\varepsilon_3) \cdot
  \varepsilon_4$. Since, of course,
  $\bar{\alpha}(\alpha\,\varepsilon_1\times
  \alpha\,\varepsilon_2\times \alpha\,\varepsilon_3) \cdot
  \varepsilon_i = 0$ for $i=1,2,3$, one gets that
  $$\bar{\alpha}(\alpha\,\varepsilon_1\times
  \alpha\,\varepsilon_2\times \alpha\,\varepsilon_3) = \Norm(\alpha)^2
  \,\varepsilon_4 = \Norm(\alpha)^2\, \varepsilon_1\times
  \varepsilon_2\times\varepsilon_3,$$
  from which the result follows.
\end{proof}

\begin{rem}
  It follows immediately from the previous result that
  when $\alpha,\beta,\gamma\in\hh$ and $u\in\hh^*$, then
  \begin{align*}
    u\,(\alpha\times\beta\times\gamma)
    &= u\alpha\times u\beta\times u\gamma,\\
    (\alpha\times\beta\times\gamma)\,u
    &=\alpha u\times \beta u\times \gamma u.
  \end{align*}
\end{rem}

The main result of this section is now easy to prove.

\begin{thm}
  \label{thm:vp_2mult}
  Given $\alpha, \beta, \gamma, \delta \in\ll$, one has
  \begin{align*}
    \alpha\beta \times \alpha \gamma \times \delta
    &= \alpha\, (\beta
      \times \gamma \times \bar{\alpha}\delta) \in \alpha\ll\\
    \beta\alpha \times \gamma\alpha \times \delta
    &= (\beta \times \gamma \times \delta\bar{\alpha})\,\alpha\in
      \ll\alpha.
  \end{align*}
\end{thm}

\begin{proof}
  Again, we will just show the first claim, the proof of the second
  being entirely analogous.

  Using the previous result, one has:
  \begin{align*}    
    \alpha\beta \times \alpha \gamma \times \delta
    & = \alpha\beta \times \alpha \gamma
      \times\alpha\alpha^{-1}\delta\\
    & =\Norm(\alpha)\, \alpha\, (\beta \times \gamma
      \times \alpha^{-1}\delta)\\
    & = \alpha\, (\beta \times \gamma
      \times \bar{\alpha}\delta),
  \end{align*}
  which finishes the proof, given explicitly the respective right
  quotient.
\end{proof}

Observe that, in particular, Theorem~\ref{thm:vp_2mult} entails that
$\Norm(\alpha) \mid \Norm(\alpha\beta \times \alpha \gamma \times
\delta)$, and hence Theorem~\ref{PallTh1} implies, when
$\Norm(\alpha)$ is odd, that
$\alpha\beta \times \alpha \gamma \times \delta$, if primitive, has a
unique, up to left associates, divisor of norm equal to the norm of
$\alpha$. However, it is not true that
$\alpha\beta \times \alpha \gamma \times \delta \in \ll\alpha$. For
example, for $\alpha = 1+i+j+2k$, $\beta = j$, $\gamma = i$ and
$\delta = 1+i+j$, one has:
$$\alpha\beta \times \alpha \gamma \times \delta \in\ll\, \bar{\alpha}\,k.$$

Notice that when
$\nu = \alpha\beta \times \alpha \gamma \times \delta$ is such that
$\Re(\nu) = 0$, which is the case when $\delta\in\RR$, then the fact
that $\nu = \alpha\lambda$ obviously entails that
$\nu = -\bar{\nu} = -\bar{\lambda}\, \bar{\alpha}$.

In many examples, the right divisor of
$\alpha\beta \times \alpha \gamma \times \delta$ is \emph{comorphic}
to $\alpha$, i.e.~the absolute values of their coordinates are the
same, up to order.  However, for $\alpha = 1+i+3j+6k$, a prime of norm
$47$, one has, taking $\delta = 1+2i+5j+7k$ (a prime above $79$):
$$\alpha \times \alpha i \times \delta = (1-2i-2j-2k)\,(2+3i+5j+3k),$$
a product of two primes with the second being factor a prime above
$47$ that is not comorphic to $\alpha$.

This raises the following problem:

\medskip

\noindent{\bf Question:} Given $\beta, \gamma, \delta\in\ll$ such that
$\alpha\beta \times \alpha \gamma \times \delta$ is primitive, can one
describe the relation between $\alpha$ and the unique, up to left
associates, right divisor of the quaternion
$\alpha\beta \times \alpha \gamma \times \delta$ with norm equal
$\Norm(\alpha)$? When it is comorphic to $\alpha$?

\medskip

Using some of what was seen above, we can prove the following.
\begin{prop}
  When $\delta$ is a unity,
  $\alpha\beta \times \alpha \gamma \times \delta \in \ll \bar{\alpha}
  \delta.$
\end{prop}

\begin{proof}
  Using the second equality in the previous Theorem, together with
  \eqref{eq:vpx1} and \eqref{eq:vpconj},
  \[
    (\alpha\beta \times \alpha \gamma \times \delta)\bar{\delta} =
    \alpha\beta\bar{\delta} \times \alpha \gamma\bar{\delta} \times 1
    = \delta\bar{\beta}\bar{\alpha} \times \delta\bar{\gamma}
    \bar{\alpha} \times 1 = (\delta\bar{\beta} \times
    \delta\bar{\gamma} \times \alpha)\, \bar{\alpha},
   \]
   which yields the result when $\Norm(\delta)=1$.
\end{proof}
%%%%%%%%%%%%%%%%%%%%%%%%%%%%%%%%%%%%%%%%%%%%%%%%%%%%%%%%%%%%%%%%%%%%%%%%
\section{Final remarks}

As pointed out in the introduction, the results presented here where
obtained while musing on a possible extension to integral quaternions
of the method of factoring an integer from two of its representations
as a sum of two squares. Some results of G. Pall, contained in
\cite{Pal}, led us to look for integral quaternions that are
orthogonal to a given one. To construct these, we turned our attention
to the vector product, just to find out that this did not yield what
we where looking for. Nevertheless, we obtained results that seem
interesting in their own way, and that led to a question that seems
worth pondering about.

%\bigskip
%%%%%%%%%%%%%%%%%%%%%%%%%%%%%%%%%%%%%%%%%%%%%%%%%%%%%%%%%%%%%%%%%%%%%%%%
\subsection*{Acknowledgements}

The research of the first author was partially supported by CMUP ({\em
  Centro de Matem\'atica da Universidade do Porto}), which is
financed by national funds through FCT ({\em Funda\c c\~ao para a
  Ci\^encia e a Tecnologia}), I.P., under the project with reference
UIDB/00144/2021. The research of the second author was partially
financed by Portuguese Funds through FCT within the Projects
UIDB/00013/2020 and UIDP/00013/2020.

%\bigskip

%%%%%%%%%%%%%%%%%%%%%%%%%%%%%%%%%%%%%%%%%%%%%%%%%%%%%%%%%%%%%%%%%%%%%%%%


\begin{thebibliography}{Mac21b}
  %
\bibitem[AA$^+$13]{abouzaid} Mohammed Abouzaid, Jarod Alper, Steve
  DiMauro, Justin Grosslight, Derek Smith, \emph{Common Left- and
    Right-Hand Divisors of a Quaternion Integer}, Journal of Pure and
  Applied Algebra 217 (2013) 779--785.
  %
\bibitem[CP12]{CoanPerng} Boyd Coan and Cherng-tiao Perng,
  \emph{Factorization of Hurwitz Quaternions}, International
  Mathematical Forum, Vol.~7, 2012, no.~43, 2143--2156.
  %
\bibitem[CK15]{CK15} Henry Cohn and Abhinav Kumar, \emph{
    Metacommutation of Hurwitz Primes}, Proceedings of the AMS 143
  (2015), no.~4, 1459--1469.
  %
\bibitem[CS03]{CeS} John H. Conway, Derek Smith, \textsc{On
    Quaternions and Octonions}, AK Peters 2003.
  % 
\bibitem[Dic92]{Dic} L. E. Dickson, \textsc{History of the Theory of
    Numbers}, AMS Chelsea Publishing, 1992.
  %
\bibitem[FG$^+$16]{forsyth} A. Forsyth, J. Gurev, and S. Shrima,
  \emph{Metacommutation as a Group Action on the Projective Line over
    $\FF_p$}, Procedings of the AMS 144 (2016), no.~11, 4583--4590.
  %
\bibitem[Hur96]{Hur96} A. Hurwitz, \emph{Ueber die Zahlentheorie der
    Quaternionen}, Nachrichten von der Gesellschaft der Wissenschaften
  zu Go\"ottingen, Mathematisch-physikalische Klasse (1896) 313--340.
  %
\bibitem[Hur19]{Hur19} A. Hurwitz, \textsc{Vorlesungen \"Uber die
    Zahlentheorie der Quaternionen}, Julius Springer, 1919.
  %
\bibitem[Lip86]{Lip} M. Lipschitz, \emph{Recherches sur la
    Transformation, par des Substituitions R\'eelles, d'une Somme de
    Deux ou de Trois Carr\'es en Elle-m\^eme}, Journal de
  Math\'ematiques Pures et Appliqu\'es 2 (1886), 373--439.
  %
\bibitem[Mac21a]{mac1} Ant\'onio Machiavelo, Nikolaos Tsopanidis,
  \emph{Zhi-Wei Sun's 1-3-5 Conjecture and Variations}, Journal of
  Number Theory 222 (2021) 1--20.
%
\bibitem[Mac21b]{mac2} Ant\'onio Machiavelo, Rog\'erio Reis, Nikolaos
  Tsopanidis, \emph{Report on Zhi-Wei Sun's ``1-3-5 Conjecture'' and
    Some of Its Refinements}, Journal of Number Theory 222 (2021)
  21--29.
  % 
\bibitem[Pal40]{Pal} Gordon Pall, \emph{On the Arithmetic of
    Quaternions}, Transactions of the AMS 47 (1940), 487--500.
  % 
\bibitem[RS86]{ReS} Michael O. Rabin and Jeffery O. Shallit,
  \emph{Randomized Algorithms in Number Theory}, Communications on
  Pure and Applied Mathematics XXXIX (1986), S239--S256.
  %
\bibitem[Sage21]{sage} The Sage Developers, \emph{SageMath, the Sage
    Mathematics Software System (Version 9.7)}, 2021,
  \url{https://www.sagemath.org}.
  % 
\bibitem[Voi22]{Voight} John Voight, \textsc{Quaternion Algebras},
  v.1.0.5, June 7, 2023, available at:
  \url{https://math.dartmouth.edu/~jvoight/quat.html}
\end{thebibliography}
\end{document}